\documentclass[reqno, 12pt]{amsart}

\numberwithin{equation}{section}

\usepackage{amssymb}
\usepackage{mathrsfs}
\usepackage{dsfont}
\usepackage{enumerate, xspace}

\usepackage{changebar}
\usepackage{hyperref}
\usepackage{pdfsync}
\usepackage{color}
\usepackage{bm}
\usepackage{amsmath}
\usepackage[a4paper, margin=10mm]{geometry}

\newtheorem{thmm}{Theorem}[section]

\newtheorem{thm}{Theorem}[section]

\newtheorem{prop}[thm]{Proposition}

\newcommand\cH{{\mathcal H}}

\newcommand\cL{{\mathcal L}}

\newcommand\bE{{\mathbb E}}

\newcommand\bP{{\mathbb P}}
\newcommand\bR{{\mathbb R}}

\newcommand\ve{\varepsilon}

\def\La{{\Lambda}}
\def\ga{{\gamma}}

\newcommand\nn{\nonumber}

\newcommand{\h}{\mathsf{h}}				

\begin{document}

\title[Quasi-Static Large Deviations]{Quasi-static Large Deviations}

\author{Anna de Masi}
\address{Anna de Masi\\Universit\`a dell'Aquila,\\
 67100 L'Aquila, Italy}
\email{{\tt demasi@univaq.it}}

\author{Stefano Olla}
\address{Stefano Olla\\
CEREMADE, UMR CNRS\\
Universit\'e Paris-Dauphine, PSL Research University\\
75016 Paris, France}
\email{{\tt olla@ceremade.dauphine.fr}}

\date{\today. {\bf File: {\jobname}.tex.}}
\begin{abstract}
We consider the symmetric simple exclusion with open boundaries that are in contact with 
particle reservoirs at different densities. The reservoir densities changes at a slower time scale 
with respect to the natural time scale the system reaches the stationary state.
This gives rise to the
quasi static hydrodynamic  limit proven in \cite{dmo1}. We study here the large deviations with respect to this limit
for the particle density field and the total current. We identify explicitely the large deviation functional 
and prove that it satisfies a fluctuation relation. 
 \end{abstract}
\thanks{We thank Davide Gabrielli and Errico  Presutti for stimulating discussions and remarks.
  We thank the kind hospitality of GSSI. We also thank the referees for the careful review and
  the stimulating remarks.\\
SO's research is supported by ANR-15-CE40-0020-01 grant LSD} 
\maketitle

\section{Introduction}
\label{sec:introduction}

In usual hydrodynamic limits it is studied the macroscopic evolution of the conserved quantities of a 
\emph{large} microscopic system, under a space-time scaling such that the time scale is the 
typical one where the system {reaches equilibrium}.
When the system is open and connected 
to thermal or particle reservoirs, the parameters of these reservoirs give the boundary conditions for the 
partial differential equations   { which describe } the macroscopic evolution of the conserved quantities.

{ In this paper we consider the case in which the parameters
of the boundary reservoirs (for example the particle density or the temperature) are not constant but change in time on a slower scale than that of relaxation at equilibrium. Rescaling the time at this slower scale
give a \emph{quasi static} macroscopic evolution for the conserved quantities: at each \emph{macroscopic} time 
the profile of the conserved quantity is equal to the stationary one corresponding to the given boundary condition at that time.}
When these stationary profiles are of equilibrium, this limit models the thermodynamic quasi static transformations, 
where the Clausius equality holds, i.e. the \emph{work}
{\color{red} is} done by the boundaries. 

In \cite {bgjll13}, \cite{olla2014} and \cite{LO2015}, quasi static transformations are obtained from a time rescaling of the
 macroscopic diffusive equation.

In \cite{dmo1} we studied,  for various  stochastic particle systems 
whose macroscopic evolution is described by a diffusive equation,
 the direct hydrodynamic quasi static limit by rescaling properly space and time in the microscopic dynamics,

 We study here the large deviation for one of these models, the symmetric simple exclusion process. 


 {
The system we consider is composed by $2N+1$ sites (denoted by $-N, \dots, N$) where particles move like symmetric 
random walks with exclusion. We add birth and death processes at the left and right  boundaries that describe the interaction with reservoirs.
These reservoirs have densities $\rho_-(t)$ on the left and $\rho_+(t)$ on the right, that are time dependent.  The time scale at which 
these density $\rho_\pm$ change defines the \emph{macroscopic time scale}.
If the particles jump and are created and destroyed at rate $N^2$, the macroscopic density profile evolves following a linear diffusive equation 
with boundary conditions $\rho_\pm(t)$. This is the usual hydrodynamic limit and for $\rho_\pm$ 
constant this result is known since the  '80's, \cite{GKMP}.}

{ In the quasi static limit particles jump (or are created and destroyed at the boundaries) 
with rate on order $N^{2+\alpha}$ with $\alpha >0$.
As proved in  \cite{dmo1}, at each time $t$,  the empirical density converges, 
as $N\to \infty$, to the solution of  the stationary heat equation with boundary conditions $\rho_\pm(t)$: }
\begin{equation}
  \label{eq:75}
  \bar \rho(t,y)=\frac 12[\rho_+(t)-\rho_-(t)]y +\frac 12\left[\rho_+(t) + \rho_-(t)\right], \qquad y\in[-1,1],
\end{equation}
and the total current through any bond will converge to 
\begin{equation}
  \label{eq:76}
  \bar J(t) = - \int_0^t \frac 12 [\rho_+(s)-\rho_-(s)] \; ds.
\end{equation}
Thus the quasi-static limit gives the evolution of macroscopic profile, completely driven by the boundary conditions.

 In \cite{dmo1}  it is also proven that 
  the distribution of the process is close to a product Bernoulli measure with parameter 
$  \bar \rho(t,y)$. 
Stronger results are proved by controlling the correlation functions with the same techniques used in \cite{DPTV2}.
First order  corrections have been studied in \cite{ch-landim}.

{ Observe that if $\rho_+(t)=\rho_-(t)$ then the above result says that the stochastic process at any macroscopic time 
 is close to equilibrium but the order parameter changes in time thus  performing a quasi static transformation.}

We study in this article  {\it {the joint }} large deviations {\it{of the density and the current} }with 
respect to this quasi-static limit.
We prove that the probability to observe a density profile  $\rho(t,y)$ that satisfies the boundary conditions 
$\rho_\pm(t)$ and a total current $J(t)$ is
asymptotically {  $\bP\left( \rho, J\right) \ \sim \ e^{-N^{1+\alpha}I(J,\rho)}$}
where 
\begin{equation}
  \label{eq:78}
   I(J,\rho) = \frac 14 \int_0^T \int_{-1}^{1} 
   \frac{\left( J'(t) +  \partial_y \rho(t,y)\right)^2}{\rho(t,y)[1-\rho(t,y)]}dy \; dt .
\end{equation}
See section \ref{sec:large-devi-theor} for a precise statement. Notice that $J(t)$ is space constant: in the quasi static limit 
the total current must be homogeneous in space. 
{Notice also that $I$ does not depend on the initial configuration}.

{ We also prove that $I$ statisfies the fluctuation relation }
\begin{equation}
 I(J,\rho) -  I(-J,\rho) =  \int_0^T J'(t) (z_+(t) - z_-(t)) dt,
\label{eq:79}
\end{equation}
where $z_\pm (t) = \log  \frac{\rho_\pm(t)}{1 - \rho_\pm(t)} $. 
{There are two interesting consequences of  \eqref{eq:79}: the first one is  that the difference $ I(J,\rho) -  I(-J,\rho) $
does not depend on $\rho$ (as long as $\rho$ satisfy the boundary conditions).
The other  is  that, since $I(\bar J,\bar \rho)=0$,

\begin{equation*}
  \begin{split}
    I(-\bar J,\bar \rho)= \int_0^T \frac 12[\rho_+(t)-\rho_-(t)] (z_+(t) - z_-(t))\; dt \\
    = \frac 12 \int_0^T \left[\mathcal H(\rho_+(t),\rho_-(t)) +  \mathcal H(\rho_-(t) ,\rho_+(t)) \right]\; dt,
  \end{split}
\end{equation*}
where $\mathcal H(\rho,\nu) = \rho\log \frac{\rho}{\nu} + (1-\rho) \log \frac{1-\rho}{1-\nu}$.}
Thus  the cost to {invert} the Fick's law is explicetely computable in terms of the  boundary conditions.



The formula \eqref{eq:78} looks similar to the \emph{fundamental formula} of the Macroscopic Fluctuation Theory
\cite{bertiniprl} \cite{bertinirmp}, but differs in some important points.  
In the usual hydrodynamic scaling, i.e. for $\alpha = 0$, the probability to observe a density profile
$\rho(t,y)$ (satisfying the boundary conditions $\rho(t,\pm 1) = \rho_\pm(t)$) and a current field $J(t,y)$, 
which must be related by the conservation law $\partial_t \rho(t,y) = -\partial_y \partial_t J(t,y)$,
 behaves like $e^{-N I_0(J,\rho) }$ with
\begin{equation}
  \label{eq:mft}
  I_{0}(J,\rho) =  \frac 14 \int_0^T \int_{-1}^{1} 
  \frac{\left(\partial_t J(t,y)  +  \partial_y \rho(t,y)\right)^2}{\rho(t,y)[1-\rho(t,y)]} dy \; dt
  + \int_{-1}^{1} \mathcal H(\rho(0,y),\nu_0(y)) dy.
\end{equation}

The second term on the right hand side of \eqref{eq:mft}
is due to the large deviations of the initial profile, if the initial distribution of the
process {is an inhomodeneous Bernoulli distribution with $\nu_0(y)$
as macroscopic profile of density}.

Formula \eqref{eq:mft} appears first in \cite{bertiniprl},
 while a time homogeneous version was previously introduced in \cite{bd04}. More precisely,
in the case $\rho_\pm$ constant in time and $\alpha =0$, the probability to observe a \emph{total} and \emph{time averaged}
current $\frac {Q(T)}{T} = \frac 1T \int_{-1}^1 J(T,y) dy  = 2q$ behaves like $e^{-N T I_{BD}(q) }$ with
\begin{equation}
  \label{eq:82}
  I_{BD}(q, \rho_+, \rho_-) =
  \inf_{\rho(\cdot): \rho(\pm 1) = \rho_\pm} \frac 14 \int_{-1}^{1}
  \frac{\left(q  + \rho'(y)\right)^2}{\rho(y)[1-\rho(y)]} dy
\end{equation}
In \cite{bd04} such large deviations behaviour is obtained under the assumption that the optimal density profile to obtain 
the \emph{total} current $q$ is independent of time. This is true in the symmetric simple exclusion case but not in all dynamics,
and the existence of time dependent optimal profile is related to the existence of dynamical phase transitions (cf. \cite{bertinijsp2016} 
\cite{bd06}). Notice that if we look at the large deviation such that $\partial_t J(t,y) = 2q$ (for all $t$ \emph{and} $y$), then 
the minimizing density must satisfy the continuity equation and consequently it is constant in time.
 For these \emph{space-homogeneous} deviations of the current field, $I_{BD}$ gives the right rate function.

In the quasistatic case, $\alpha>0$, the current field $J(t,y)$ has to be constant in $y$: 
 we cannot have large deviations where the intensity of the current is not constant in space, {see \eqref{2.10}.
Consequently there is no continuity equation to be satisfied between $J(t)$ and $\rho(t,y)$. 

This gives a direct connection with the Bodineau-Derrida functional $I_{BD}$:
 minimising the rate function $I(J,\rho)$ defined in \eqref{eq:78} over $\rho(t,y)$ in order to obtain  the large deviation in the quasistatic limit of the current $J$ we have (cf. \eqref{eq:73} at the end of section \ref{sec:rate-function})
\begin{equation}
  \label{eq:83}
  I(J) = \int_0^T  I_{BD}\left(J'(t), \rho_+(t), \rho_-(t)\right)\; dt.
\end{equation}



For simplicity we have restricted our attention to the symmetric simple exclusion, but in principle 
the result can be extended to other dynamics (like weakly asymmetric exclusion, zero range, KMP models etc.).
  Since in the symmetric simple exclusion $I_{BD}(J)$ is convex,
  we do not expect in the quasi static case the existence of \emph{dynamical phase transitions} 
  (cf. \cite{bertiniprl}, \cite{bertinijsp2016}).
In the other dynamics this remain an interesting question to be investigated.  

The main scheme of the proof goes along the lines of \cite{KOV1989} and \cite{BLM},  
but with the further feature of controlling also the deviations  of the current. 
Using a variational characterization of $I$, for deviations such that $I(J,\rho)<+\infty$, it is possible 
to find suitable regular approximations $J_\ve$ and $\rho_\ve$ such that $I(J_\ve,\rho_\ve) \to I(J,\rho)$. 
For regular $J,\rho$, it is possible to construct the weakly asymmetric exclusion dynamics {as described above} 
such that the corresponding 
quasi-static limit is given by $J$ and $\rho$ and whose relative entropy with respect to the original process converges 
to $I(J,\rho)$. This takes care of the lower bound. For the upper bound we use 
suitable exponential martingales that control also the current, and a superexponential estimate adapted from 
the original idea in \cite{KOV1989}.

\section{Simple exclusion with boundaries}
\label{sec:simple-excl-with}

We consider the exclusion process in $\{0,1\}^{\La_N}$,
$\La_N:=\{-N,..,N\}$ with reservoirs at the boundaries with density  
$\rho_\pm(t)\in [a,1-a]$ for some $a>0$. {
We assume that $\rho_\pm(t)\in C^1$}.

Denoting by $\eta(x)\in  \{0,1\}$ the occupation number at $x\in
\La_N$ we define the dynamics via the generator 
		\begin{equation}
		\label{0.1}
 L_{N,t}=N^{2+\alpha}[L_{\text{exc}}+L_{b,t}],\qquad t\ge 0,\quad \alpha>0
			\end{equation}
where for a given $\ga>0$, {
 \begin{equation}
   \label{0.2}
L_{\text{exc}}f(\eta) 
= \gamma \sum_{x=-N}^{N-1} \Big(f(\eta^{(x,x+1)})-f(\eta)\Big)=:
  \sum_{x=-N}^{N-1}\nabla_{x,x+1}f(\eta)
   \end{equation}
$\eta^{(x,y)}$ is the configuration obtained from $\eta$ by exchanging the occupation numbers at $x$ and $y$}, and  
\begin{equation}
  \begin{split}
    L_{b,t}f(\eta) = \sum_{\j = \pm}\rho_{\j}(t)^{1-\eta(\j N)} (1-\rho_{\j}(t))^{\eta(\j N)}
    [f(\eta^{\j N})-f(\eta)] 
  \end{split}
 \label{eq:1}
\end{equation}
where $\eta^x (x) = 1-\eta(x)$, and $\eta^x (y) = \eta(y)$ for $x\neq y$.

We recall the quasi-static hydrodynamic limit proven in \cite{dmo1}:
	\begin{thm}
	\label{thm1}
For any $\alpha>0$ and  for  any macroscopic time $t> 0$ the following holds. For any initial configuration $\eta_0$, for any $y\in[-1,1]$ and for any local function $\varphi$
	\begin{equation}
	\label{0.5}
\lim_{N\to\infty}\mathbb E_{\eta_0}\Big(
\theta_{[Ny]} \varphi (\eta_{t})\Big)=
 <\varphi(\eta)>_{\bar \rho(y,t)} =: \hat\varphi(\bar \rho(y,t)) 
	\end{equation}
where $[\cdot]$ denotes integer part, $\theta$ is the shift operator, $<\cdot>_{\rho}$
 is the expectation with respect to the product Bernoulli measure of density $\rho$, 
and  
\begin{equation}
 \label{eq:2} 
 \bar \rho(y,t)=\frac 12[\rho_+(t)-\rho_-(t)] y
 +\frac 12\left[\rho_+(t) + \rho_-(t)\right], \qquad y\in[-1,1]
 \end{equation}
is the quasi-static profile of density at time $t$.
{In particular}  for any $t>0$ and any $y\in[-1,1]$
\begin{equation}
	\label{1.11}
\lim_{N\to\infty}\,\,
\mathbb E_{\eta_0}\big[\eta_{t}([Ny])\big] \ =\ \bar\rho(y,t)
	\end{equation}
	\end{thm}

In the following we will use the notation
\begin{equation}
  \label{eq:41}
  \phi(\rho) = \rho(1-\rho).
\end{equation}

Define the following counting processes: for $x=-N-1, \dots, N$ 
\begin{equation}
  \label{eq:18}
  \begin{split}
{  \h_+(t,x)}
    &= \{\text{number of jumps $x \to x+1$ up to time $t$}\},\\
{   \h_-(t,x) } &= \{\text{number of jumps $x+1 \to x$ up to time $t$}\}, \\
    \h(t,x)&= \h_+(t,x) - \h_-(t,x) 
  \end{split}
\end{equation}

When $x = -N-1$ the corresponding $\h_+(t,-N-1)$ is the number of particles that enters on the left boundary, and 
$\h_+(t,N)$ is the number of particles that exit at the right boundary.

The conservation law is microscopically given by the relation
\begin{equation}\label{eq:claw}
  \eta_t(x) - \eta_0(x) = \h(t,x-1) - \h(t,x), \qquad x= -N , \dots, N
\end{equation}
  Furthermore we have that for $x=-N,\dots,N-1$:
\begin{equation*}
   \h(t,x) = \gamma N^{2+\alpha} \int_0^t \left(\eta_s(x) - \eta_s(x+1)\right) ds +M(t)
 \end{equation*}
 where $M(t)$ is a martingale.


For $y\in [-1,1]$, define
$$
\h_N(t,y) = \frac 1{N^{1+\alpha}} \h(t, [Ny]).
$$

Notice that, for any $x,x'\in \{-N,\dots,N\}$, by \eqref{eq:claw} we must have
 $|\h(t,x) - \h(t,x')| \le 2N$, that implies
\begin{equation}
\label{2.10}
  |\h_N(t,y) - \h_N(t,y')| \le \frac{2}{N^{\alpha}}, \quad \forall y, y' \in [-1,1].
\end{equation}

It follows that
\begin{equation*}
  \begin{split}
   \frac1{N^{1+\alpha}} \mathbb E_{\eta_0}\Big(\h(t,x)\Big) \ =\  \frac1{N^{1+\alpha}} \frac 1{2N}
   \mathbb E_{\eta_0}\Big(\sum_{x'=-N}^{N-1} \h(t,x')\Big)  + O(N^{-\alpha})
   \\
   = \frac1{N^{1+\alpha}} \frac 1{2N}
   \mathbb E_{\eta_0}\Big(\gamma N^{2+\alpha} \int_0^t \sum_{x'=-N}^{N-1} (\eta_s(x') - \eta_s(x'+1))  \Big)  + O(N^{-\alpha})
 \\ = \frac \gamma 2 \mathbb E_{\eta_0}\Big(\int_0^t \left(\eta_s(-N) - \eta_s(N)\right) ds\Big) 
    + O(N^{-\alpha})
  \end{split}
\end{equation*}
	Thus from \eqref{1.11}
\begin{equation}
  \label{eq:14}
  \begin{split}
    &\lim_{N\to\infty}\mathbb E_{\eta_0}\Big(\h_N(t,y)\Big) \ =\ \bar h(t) \ :=  \ - \frac \gamma 2\int_0^t [\rho_+(s)-\rho_-(s)] ds.
  \end{split}
\end{equation}
{Observe that the total current depends on $\ga$ 
while the quasi-static profile does not. In the sequel we will set $\ga=1$.}

\section{The rate function}
\label{sec:rate-function}



We denote by $\mathcal M =\left\{ \rho(t,y) \; \text{measurable}, t\in [0,T], y\in [-1,1], 0\le \rho(t,y) \le 1\right\}$. 
We endow $\mathcal M$ of the weak topology, i.e. for any continuous function $G \in \mathcal C\left([0,T]\times [-1,1]\right)$,
$\rho \mapsto \int_{0}^T dt \int_{-1}^1 dy \rho(t,y) G(t,y)$ is continuous on $\mathcal M$. 
Note that $\mathcal M$ is compact under this topology.

Let $\rho(t,y) \in \mathcal M$ and  $J(t) \in \mathcal D\left([0,T], \bR \right)$, 
with $J(0) = 0$.
The rate function is defined by:  {
\begin{equation}
  I(J,\rho)  =  \sup_{H\in \mathcal C^{1,2}([0,T]\times[-1,1]),\atop H(\cdot,-1) = 0} \Big\{ \cL(H;J,\rho)  - \int_0^T dt \int_{-1}^{1} dy \left(\partial_y H(t,y)\right)^2 \phi(\rho(t,y))\Big\}
\label{eq:24}
 \end{equation}
where
 \begin{eqnarray}
\nn
&& \hskip-.7cm\cL(H;J,\rho):=
H(1,T) J(T) + \int_0^T dt \Big(-\partial_t H(t,1) J(t)   
  - \int_{-1}^{1}\partial_{yy} H(t,y) \rho(t,y) dy
   \\&& \hskip1.6cm  + \left(\partial_yH(t,1) \rho_+(t) - \partial_yH(t,-1) \rho_-(t)\right) \Big) 
   \label{eq:24b}
\end{eqnarray}
}

Observe that if $\partial_y \rho$ and $J'$ exist and are regular enough and $\rho(t, \pm 1) = \rho_{\pm}(t)$, { then
$ I(J,\rho) $ is given by
\begin{equation}
  \label{eq:rate}
   I(J,\rho) = \frac 14 \int_0^T \int_{-1}^{1} 
   \frac{\left( J'(t) +  \partial_y \rho(t,y)\right)^2}{\phi(\rho(t,y))} dy \; dt ,
\end{equation}
and 
the maximum is reached on the function
\begin{equation}
\label{eq:barh}
  \bar H(t,y) = \frac 12 \int_{-1}^y \frac{J'(t) + \partial_y \rho(t,y')}{\phi(\rho(t,y'))} dy'.
\end{equation}} 


Define 
\begin{equation*}
  \begin{split}
    \cH_+ = \text{clos}\Big\{ &H(t,y)\in \mathcal C^{1,2}([0,T]\times[-1,1]), H(\cdot,-1) = 0 : \\
      &\|H\|_{\rho, +}^2 = \int_0^T dt \int_{-1}^{1} dy \left(\partial_y H(t,y)\right)^2
      \phi(\rho(t,y)) < +\infty\Big\}.
  \end{split}
\end{equation*}
and the dual space
\begin{equation*}
  \begin{split}
    \cH_{-} = \text{clos}\Big\{ &H(t,y)\in \mathcal C^{1,2}([0,T]\times[-1,1])
    : \\
      &\|H\|_{\rho,-}^2 = \int_0^T dt \int_{-1}^{1} dy \left(\partial_y H(t,y)\right)^2
      \phi(\rho(t,y))^{-1} < +\infty\Big\}.
  \end{split}
\end{equation*}


\begin{prop}\label{prop1}
  If $ I(J,\rho) <\infty$ then the weak derivatives $J'(t)$ and $\partial_y \rho(t,y)$ exists in $ \cH_{-}$ and $\rho(t, \pm 1) = \rho_{\pm}(t)$,
furthermore
  \begin{equation}
    \label{eq:50}
     I(J,\rho) = \frac 14 \int_0^T \int_{-1}^{1}  \frac{ \left( J'(t) + \partial_y \rho(t,y)\right)^2}{\phi(\rho(t,y))}  \; dy \; dt .
  \end{equation}
\end{prop}

\begin{proof}
Choose $H(t,y) = z(t) (1+y)$ for a given smooth function $z(t)$ on $[0,T]$ in the variational formula 
\eqref{eq:24}. Then defining 
\begin{equation}
  \label{eq:15}
  \begin{split}
    \mathcal Q(J,\rho) = \sup_{z\in \mathcal C^1([0,T])} \Big\{ &2 z(T) J(T)  - \int_0^T \left(2 z'(t) J(t) dt + z(t)
        (\rho_+(t)- \rho_-(t)) \right)\\
    &  - \int_0^T z(t) ^2 \int_{-1}^1\phi(\rho(t,y)) dy \Big\}\\
    & := \sup_z \Big\{ \tilde L(J,z)  - \int_0^T dt \; z(t)^2 \int_{-1}^1\phi(\rho(t,y)) dy \Big\}
  \end{split}
\end{equation}
we have that $\mathcal Q(J,\rho) \le I(J,\rho) < +\infty$. 
It follows that
\begin{equation*}
  |\tilde L(J,z)|^2 \le 4 \mathcal Q(J,\rho) \int_0^Tdt\;  z(t)^2 \int_{-1}^1\phi(\rho(t,y)) dy.
\end{equation*}
This means that $z \to \tilde L(J,z)$ is a bounded linear functional on the Hilbert space $\mathcal N_+$, where
\begin{equation*}
   \mathcal N_{\pm} = \text{clos} \left\{ z\in \mathcal C^1([0,T]) : \|z\|^2 = \int_0^T  z(t)^2 \bar\phi(t)^{\pm 1} dt \right\}
\end{equation*}
where $\bar\phi(t) = \int_{-1}^1\phi(\rho(t,y)) dy$.
Applying Riesz representation theorem, there exists a function $g(t)\in  \mathcal N_+$ such that
\begin{equation*}
  \tilde L(J,z) = \int_0^T  g(t) z(t)  \bar\phi(t) dt.
\end{equation*}
This implies the existence of the weak derivative  $J'(t)$ such that 
$2 J'(t) + (\rho_+(t)- \rho_-(t)) \in \mathcal N_-$. Furthermore
$g(t) \bar\phi(t) = 2 J'(t) + (\rho_+(t)- \rho_-(t))$ and 
\begin{equation}
  \label{eq:31}
  \mathcal Q(J,\rho) = \frac 14 \int_0^T \left[ 2 J'(t) + (\rho_+(t)- \rho_-(t)) \right]^2\bar\phi(t)^{-1} \; dt.
\end{equation}

  The linear functional $\cL(H;J,\rho) $ {defined in \eqref{eq:24b}} is bounded by
  \begin{equation}
    \label{eq:46}
     \left| \cL( H; J,\rho) \right|^2 \le 4 I(J,\rho) 
  \int_0^T dt \int_{-1}^{1} dy \left(\partial_y H(t,y)\right)^2 \phi(\rho(t,y)), 
  \end{equation}
and it can be extended to a bounded linear functional on ${\cH_+} $.
By Riesz representation there exists a function $G\in {\cH_+} $ such that
\begin{equation}\label{eq:riesz}
  \cL(H; J,\rho) = \int_0^T dt \int_{-1}^{1} dy \; \partial_y  H(t,y) \partial_y G(t,y) \; \phi(\rho(t,y)). 
\end{equation}

Since we have already proven that $I(J,\rho) < +\infty$ implies the existence of the weak derivative 
$J'(t) \in \mathcal N_{-}$, 
we can rewrite the variational formula {\eqref{eq:24} as }
\begin{equation}
  \label{eq:24-b}
  \begin{split}
    I(J,\rho) =& +\infty\qquad 
    \text{{if $J(t)$ is not differentiable.}}\\
    I(J,\rho) = &\sup_{H\in \mathcal C^{1,2}([0,T]\times[-1,1]),\atop H(\cdot,-1) = 0}  
    \Big\{ \int_0^T dt H(t,1) J'(t) \\
    &+  \int_0^T dt \left(\partial_yH(t,1) \rho_+(t) - \partial_yH(t,-1) \rho_-(t) \right)  \\
    &-  \int_0^T dt \int_{-1}^{1} dy     \left[ \partial_{yy} H(t,y) \rho(t,y) + \left(\partial_y H(t,y)\right)^2 \phi(\rho(t,y))\right]
    \Big\} \\
     =& \sup_H \Big\{ \cL(H;J,\rho) - \int_0^T dt \int_{-1}^{1} dy \left(\partial_y H(t,y)\right)^2 \phi(\rho(t,y))\Big\}.
  \end{split}
\end{equation}
Notice that we can rewrite
\begin{equation*}
  \begin{split}
    \cL(H;J,\rho) = - \int_0^T dt \Big[\int_{-1}^1 \partial_{yy} H(t,y) \left(J'(t) y + \rho(t,y)\right) \; dy\\
    {\color{red}+} \partial_yH(t,1) \big(\rho_+(t) + J'(t)\big) - \partial_yH(t,-1) \big(\rho_-(t) - J'(t)\big) \Big]
  \end{split}
\end{equation*}
{Let us first show that $I(J,\rho) <\infty$ implies that $\rho(t, \pm 1) = \rho_{\pm}(t)$, in the sense that,
for any continuous function $g(t)$ on $[0,T]$,
  \begin{equation*}
    \begin{split}
    \lim_{\delta\to 0} \frac 1{\delta} \int_0^{T} dt \int_{-1}^{-1+\delta} dy \rho(t,y) g(t) = \int_0^T \rho_{-}(t) g(t) dt,
    \\
  \lim_{\delta\to 0} \frac 1{\delta} \int_0^{T} dt \int_{1-\delta}^1 dy \rho(t,y) g(t)  = \int_0^T \rho_{+}(t) g(t) dt.
\end{split}
  \end{equation*}
In fact assume that these boundary conditions are not satisfied, and choose the functions $H(s,y)$ such that 
$$
\partial_y H(t,y) =  A\left( 1 - \frac{1-y}{\delta} \right) 1_{[1-\delta,1]}(y) g(t)
$$
This function is not smooth, but it can smoothened up by some convolution without changing the argument.
On this function we have
\begin{equation*}
  \begin{split}
   \cL(H;J,\rho) = -\frac A{\delta} \int_0^{T} dt \; g(t) \int_{1-\delta}^1 dy  \left(J'(t) y + \rho(t,y)\right) 
   +  A \int_0^{T} g(t) \big(\rho_+(t) + J'(t)\big) \; dt \\
   =  \frac 12 A{\delta} \int_0^{T} g(t) J'(t) dt -  A \int_0^{T} dt g(t) \left(\delta^{-1} \int_{1-\delta}^1 dy \; \rho(t,y) -  \rho_+(t)\right)
 \end{split}
\end{equation*}
while we have
\begin{equation}
  \label{eq:84}
  \begin{split}
   \int_0^T dt \int_{-1}^{1} dy \left(\partial_y H(t,y)\right)^2 \phi(\rho(t,y)) \\
   = A^2 \int_0^{T} dt\; g(t)^2 \int_{1-\delta}^1 dy \left( 1 - \frac{1-y}{\delta} \right)^2 \phi(\rho(t,y))\\
   \le C A^2 T \delta.
 \end{split}
\end{equation}
Then if the boundary condition in $y=1$ is not satisfied, one can construct a sequence of functions 
$H$ such that $I(J,\rho) = +\infty$. The other side is treated in a similar way.
}

Thus  $I(J,\rho) <\infty$ 
implies that the boundary conditions must be satisfied and that there exist
the weak derivative $\partial_x \rho(t,x)$ and from \eqref{eq:riesz} 
we can identify $\partial_x G(t,x) \phi(\rho(t,x)) = J'(t) + \partial_x \rho(t,x)$.

Notice that, from this identification we can rewrite $I$ as 
\begin{equation}
  \label{eq:70}
  \begin{split}
    I(J,\rho) = \frac 14 \int_0^T \int_{-1}^1  \frac{J'(t)^2
      + \partial_y \rho(t,y)^2}{ \phi(\rho(t,y))} dy\; dt\\
    + \frac 12 \int_0^T  J'(t)  \log\left(\frac{\rho_+(t)(1-\rho_-(t))}{\rho_-(t)(1-\rho_+(t))}\right) dt
  \end{split}
\end{equation}
that proves that $J$ and $\rho$ are in $\cH_-$, since $\rho_\pm(t)$ are assumed bounded away from $0$ and $1$.
\end{proof}

\begin{prop}\label{prop2}
   If $ I(J,\rho)  < \infty$, there exists an approximation by bounded functions $J_\ve(t)$ and $x$-smooth  
   $\rho_\ve(t,x)$ of $J$ and $\rho$, such that $\rho_\ve(t, \pm 1) = \rho_{\pm}(t)$ and
 \begin{equation}
   \label{eq:45}
   \lim_{\ve\to 0} I(J_\ve, \rho_\ve) = I(J,\rho).
 \end{equation}
\end{prop}

\begin{proof}
We  follow a similar argument used in \cite{BLM}.


We proceed in two steps. We first define an approximation $\tilde\rho_\ve(t,x)$ bounded away from $0$ and $1$ 
such that $\phi(\tilde\rho_\ve(t,x)) \ge C\ve^2$ and 
$I(J,\tilde\rho_\ve) \longrightarrow I(J,\rho)$. Then assuming that $\phi(\rho(t,x))\ge a>0$, we construct a 
smooth approximation $\rho_\ve(t,x)$ such that $I(J,\rho_\ve) \longrightarrow I(J,\rho)$.
The conclusion follows then by a diagonal argument.

For the first step, define (recall \eqref{eq:2})
\begin{equation}
  \label{eq:49}
  \tilde\rho_\ve(t,x) = (1-\ve) \rho(t,x) + \ve \bar\rho(t,x).
\end{equation}
Notice that $\tilde\rho_\ve(t,\pm 1) = \rho_{\pm}(t)$ and 
$$
\phi(\tilde\rho_\ve(t,x)) \ge \ve^2 \left(\rho_{-}(t) \wedge \rho_{+}(t)\right)\left( 1 - \left(\rho_{-}(t) \vee \rho_{+}(t)\right)\right) .
$$
{$I(J,\rho)$ is convex and lower semicontinuous in $\rho$
  since it is a $\sup$  of continuous and convex functions of $\rho$.} Consequently we have
\begin{equation}
  \label{eq:51}
  I(J,\tilde\rho_\ve) \ \le \ (1-\ve)I(J,\rho) + \ve I(J,\bar\rho) \ 
  \mathop{\longrightarrow}_{\ve\to 0}\  I(J,\rho),
\end{equation}
while by lower semicontinuity
\begin{equation}
  \label{eq:52}
  \liminf_{\ve \to 0}  I(J,\tilde\rho_\ve) \ \ge \  I(J,\rho),
\end{equation}
that concludes the first approximation step.

Now assuming that $\phi(\rho(t,x))\ge \phi(a)>0$, i.e. $a< \rho(t,x) < 1-a$, we
 construct a smooth approximation $\rho_\ve(t,x)$ such that
\begin{equation*}
  \rho_\ve(t,x) \ \mathop{\longrightarrow}_{\ve\to 0} \rho(t,y), \qquad  \rho_\ve(t,\pm 1) = \rho_{\pm}(t).
\end{equation*}
and such that $I(J,\rho_\ve) \longrightarrow I(J,\rho)$.

Let $\Delta_D$ be the laplacian on $[-1, 1]$ with Dirichlet boundary conditions, and
\begin{equation*}
  R_\ve^D(x,y) = \left(I - \ve \Delta_D\right)^{-1}(x,y).
\end{equation*}
Then define
\begin{equation}
  \label{eq:37}
  \rho_\ve(t,x) \ = \bar\rho(t,x) + \int_{-1}^1  R_\ve^D(x,y) \left(\rho(t,y) - \bar\rho(t,y) \right) dy.
\end{equation}
{
%
We next prove }that there exist $a_*>0$ such that
\begin{equation}
  \label{eq:53}
  \phi( \rho_\ve(t,x)) \ge \phi(a_*).
\end{equation}

In fact notice that 
$0\le \bar\rho(t,x) - \int_{-1}^1  R_\ve^D(x,y)  \bar\rho(t,y) dy \le 1$, 
that implies
\begin{equation*}
   \rho_\ve(t,x) \ \ge \ \int_{-1}^1  R_\ve^D(x,y) \rho(t,y) dy \ge a \int_{-1}^1  R_\ve^D(x,y) dy \ge a',
\end{equation*}
for some positive $a'<a$.

Similarly we have
\begin{equation*}
  \begin{split}
    1- \rho_\ve(t,x) \ = 1 - \bar\rho(t,x) -  \int_{-1}^1 R_\ve^D(x,y) (1- \bar\rho(t,y) dy + \int_{-1}^1 R_\ve^D(x,y) (1-\rho(t,y)) dy
    \\
    \ge \int_{-1}^1 R_\ve^D(x,y) (1-\rho(t,y)) dy 
    \ge (1-a)  \int_{-1}^1 R_\ve^D(x,y) dy \ge (1-a'').
  \end{split}
\end{equation*}
Then choosing $a_* = a' \wedge a''$ we obtain \eqref{eq:53}.


Since $\rho_\ve$ is smooth, by \eqref{eq:50}, we have:
\begin{equation}
  \label{eq:40}
   I(J,\rho_\ve) = \frac 14 \int_0^T dt \int_{-1}^1 dy\ \left(J'(t) + \partial_y \rho_\ve(t,y)\right)^2 \phi(\rho_\ve(t,y))^{-1}
\end{equation}

Let $R_\ve^N(x,y) = \left(I - \ve \Delta_N\right)^{-1}(x,y)$ the resolvent 
for the laplacian with Neumann boundary conditions,
then we have the property that
\begin{equation}
  \label{eq:38}
  \partial_x  R_\ve^D(x,y) = - \partial_y R_\ve^N(x,y),
\end{equation}
that implies
\begin{equation}
  \label{eq:39}
  \partial_x  \rho_\ve(t,x) \ =  \int_{-1}^1  R_\ve^N(x,y) \partial_y \rho(t,y) dy :=  (R_\ve^N \partial_y \rho)(t,x),
\end{equation}
because $R_\ve^N(x,y)$ is a probability kernel: $\int_{-1}^1 R_\ve^N(x,y) dy = 1$.

Since 
$R_\ve^N(x,y)\le 1$ and symmetric, by convexity we have
\begin{equation}
  \label{eq:42}
  \begin{split}
    \left(J'(t) + \partial_y \rho_\ve(t,y)\right)^2 \le \int_{-1}^1 R_\ve^N(x,y)
    \left(J'(t) + \partial_x \rho(t,x)\right)^2 dx\\
    \le \int_{-1}^1 \left(J'(t) + \partial_x \rho(t,x)\right)^2 dx \le I(J,\rho) 
  \end{split}
\end{equation}
Then we have 
\begin{equation}
  \label{eq:44}
  \frac{\left(J'(t) + \partial_y \rho_\ve(t,y)\right)^2}{ \phi(\rho_\ve(t,y))} \le \frac 1{\phi(a_*)} I(J,\rho) 
\end{equation}
Then by dominated convergence we have
\begin{equation*}
   \lim_{\ve\to 0}   \mathcal I(J,\rho_\ve) \ = \  \mathcal I(J,\rho).
\end{equation*}

It remains to prove the $J'$ approximation
 by bounded functions. Let us define
   \begin{equation}
 J_k(t) = \int_0^t J'(s) \wedge{k} \vee {(-k)} \ ds \label{eq:60}
\end{equation}
and let us assume that $\rho(t,x)$ is smooth in $x$ and bounded away from $0$ and $1$. Then it is clear that 
\begin{equation*}
   \frac{\left(J'(t) \wedge{k} \vee {(-k)}  + \partial_y \rho(t,y)\right)^2}{ \phi(\rho(t,y))} 
\end{equation*}
is not decreasing in $k$ for $k$ large enough. Then by monotone convergence  $I(J_k, \rho) \to I(J,\rho)$.
\end{proof}

Notice that the minimum of $I(J,\rho)$ is correctly achieved for 
 $J'(t) = \bar J'(t) = -\frac 12\left[ \rho_+(t) - \rho_-(t) \right] $ and $\rho(t,y) = \bar\rho(t,y)$.

Furthermore, from \eqref{eq:70}, we have that for any given $\rho$, $I(J,\rho)$ is a quadratic function of $J'$.
In particular we have the following Gallavotti-Cohen type of symmetry:
\begin{equation}
  \label{eq:72}
  I(-J,\rho) = I(J,\rho) - \int_0^T J'(t) \log \frac{\rho_+(t) (1 - \rho_-(t))}{\rho_-(t) (1 - \rho_+(t))} dt, 
\end{equation}
The rate function {for $\h_N$ alone is obtained, by contraction principle (cf. \cite{Var})}, minimizing over all possible $\rho$ 
satisfying the given boundary conditions:
\begin{equation}
  \label{eq:73}
  I(J) = \inf_{\rho(t,y): \rho(t,\pm 1) = \rho_\pm(t) } I(J,\rho),
\end{equation}
and also satisfy the symmetry relation
\begin{equation}
  \label{eq:74}
  I(-J) = I(J) - \int_0^T J'(t) \log \frac{\rho_+(t) (1 - \rho_-(t))}{\rho_-(t) (1 - \rho_+(t))} dt. 
\end{equation}

Since in \eqref{eq:70} there is no relation between $J(t)$ and $\rho(t,y)$, it is possible to exchange the $\inf_\rho$
 with the time integral and obtain 
\begin{equation}
  \label{eq:73}
  I(J) 
  = \int_0^T dt \inf_{\rho(y): \rho(\pm 1) = \rho_\pm(t) } 
  \int_{-1}^1 \frac{(J'(t) + \rho'(y))^2}{\phi(\rho(y))} dy,
\end{equation}
that proves \eqref{eq:83}. 



\section{The Large Deviation Theorem}
\label{sec:large-devi-theor}

Let us define the empirical density profile 
\begin{equation}
  \label{eq:48}
 {\pi_N(t,y) = \sum_{x=-N}^{N-1} \eta_x(t) 1_{[x,x+1)}(Ny) },    \qquad y\in[-1,1]
\end{equation}
that has values on $\mathcal M$, 
so that the couple $(\h_N(1), \pi_N)$ has values on $\Xi = \mathcal D\left([0,T], \bR\right)\times \mathcal M$. 

\begin{thmm}
\label{teo4.1}
  Under the dynamics generated by \eqref{0.1}, starting with an arbitrary configuration $\eta$,
 the couple $(\h_N(1), \pi_N)$ satisfy a large deviation principle 
with rate function $I(J,\rho)$, i.e.
\begin{itemize}
\item For any closed set $C \subset \Xi$
  \begin{equation}
    \label{eq:49u}
    \limsup_{N\to\infty} \frac{1}{N^{1+\alpha}} \log \mathbb P_\eta\left( (\h_N(1), \pi_N) \in C\right) \le -\inf_{(J,\rho)\in C} I(J,\rho),
  \end{equation}
\item For any open set $\mathcal O \subset \Xi$
  \begin{equation}
    \label{eq:49l}
    \liminf_{N\to\infty} \frac{1}{N^{1+\alpha}} \log \mathbb P_\eta\left( (\h_N(1), \pi_N) \in \mathcal O\right) 
    \ge -\inf_{(J,\rho)\in \mathcal O} I(J,\rho),
  \end{equation}
\end{itemize}

\end{thmm}

\section{The superexponential estimate}
\label{sec:super-estim}

 One of the main steps in the proof of Theorem \ref{teo4.1} is  a super-exponential estimate which allows the replacement of local functions by functionals of the empirical density. 
Define
\begin{equation}
  \label{eq:25}
  V_{N,\ve}(t,\eta) = \sum_{x=-N}^{N-1} G(t,x/N) \left( \eta(x)(1-\eta(x+1)) - \phi(\bar\eta_{N,\ve}(x/N))\right)
\end{equation}
The local averages are defined in the bulk as
\begin{equation}
  \label{eq:26}
  \bar\eta_{N,\ve}(x/N) = \frac{1}{2N\ve+1}\sum_{|x'-x|\le \ve N} \eta(x'), \qquad |x| < N(1-\ve),
\end{equation}
and at the boundaries as
\begin{equation}
  \label{eq:27}
  \begin{split}
    \bar\eta_{N,\ve}(x/N) = \rho_+(t) 
     \qquad x\ge N(1-\ve),\\
   \bar\eta_{N,\ve}(x/N) =  \rho_-(t) 
   \qquad x\le -N(1-\ve).
  \end{split}
\end{equation}
{ Notice that $V_{N,\ve}(t,\eta)$ depends on $t$ not only by the function $G(t,y)$ but also by the special definition \eqref{eq:27}.}
 
In the following we use as reference measures
the inhomogeneous product measures
\begin{equation}
 \label{eq:3} 
 \mu_t(\eta)=\prod_{x=-N}^N\,  \bar \rho\left(\frac xN,t\right)^{\eta(x)}
 \left[1- \bar \rho\left(\frac xN,t\right)\right]^{1-\eta(x)}.
\end{equation}
 Observe that, since $\rho_\pm(t)$ are uniformly away from 0 and 1,  there is $C>0$ so that 
 \begin{equation}
\sup_t \sup_\eta \mu_t(\eta)\ge e^{-CN}.
\label{eq:80}
\end{equation}

\begin{prop} 
For any $\delta>0$ and initial configuration $\eta$
  \begin{equation}
    \label{eq:28}
    \lim_{\ve\to 0} \limsup_{N\to\infty} \frac 1{N^{1+\alpha}} \log 
    \bP_\eta\left(\left|\int_0^T V_{N,\ve}(t,\eta_t) dt \right| \ge N\delta\right) = -\infty.
  \end{equation}
\end{prop}

\begin{proof}
By \eqref{eq:80}, it is enough to prove \eqref{eq:28} for the dynamics with initial distribution given by $\mu_0$. 
By exponential Tchebychev inequality we have for any $a>0$:
\begin{equation*}
  \begin{split}
    \bP_{\mu_0}\left(\left|\int_0^T V_{N,\ve}(t,\eta_t) dt\right| \ge N\delta\right) \le
    e^{-N^{1+\alpha}\delta a} 
 \bE_{\mu_0} \left( e^{a N^{\alpha}\left|\int_0^T V_{N,\ve}(t, \eta_t) dt\right|}\right)
\end{split}
\end{equation*}
By using that $e^{|x|} \le e^x +e^{-x}$, all we need to prove is that there exists $K < +\infty$ such that for all $a \in \mathbb R$
\begin{equation*}
  \limsup_{\ve\to 0} \limsup_{N\to\infty} \frac 1{N^{1+\alpha}} \log \bE_{\mu_0} \left( e^{N^{\alpha}\int_0^T aV_{N,\ve}(t, \eta_t) dt}\right) \le KT
\end{equation*}
To simplify notations, denote $\tilde V(t,\eta) = N^{\alpha} aV_{N,\ve}(t, \eta)$.
Consider the equation 
\begin{equation}
  \label{eq:8}
  \partial_s u(\eta,s) =  L_{N,T-s} u (\eta,s) + \tilde V(T-s,\eta) u(\eta,s), \qquad u(\eta,0) = 1,\quad 0\le s\le T.
\end{equation}
By Feynman-Kac formula
\begin{equation}
  \label{eq:6}
  u(\eta,T) = \mathbb E_\eta \left( e^{\int_0^T \tilde V(s, \eta_s) ds}\right) 
\end{equation}
Then
\begin{equation}
  \label{eq:9}
  \begin{split}
   \frac d{ds} \frac 12 \sum_{\eta} &u(\eta,s)^2 \mu_{T-s}(\eta) =
\frac 12 \sum_\eta u(\eta,s)^2 \frac {d\mu_{T-s}(\eta)}{ds} \\
&+ \sum_{\eta}  \left(u(\eta,s)
      L_{N,T-s} u(\eta,s) +\tilde V(T-s,\eta) u^2 (\eta,s)\right) \; \mu_{T-s}(\eta) 
 \end{split}
\end{equation}
and, since $\left|\frac {d\mu_{T-s}(\eta)}{ds}\right| \le CN \mu_{T-s}(\eta)$, this is bounded by 
\begin{equation*}
    \le \left(CN + \Gamma(s) \right) \sum_\eta u(\eta,s)^2 \mu_{T-s}(\eta),
\end{equation*}
where, setting $\|f\|_s^2 = \sum_\eta f(\eta)^2 \mu_{T-s}(\eta)$,
\begin{equation}
  \label{eq:30}
  \Gamma(s) = \sup_{f, \|f\|_s = 1} \left\{ \sum_\eta \tilde V(T-s, \eta) f^2(\eta) \mu_{T-s}(\eta) 
    + \sum_\eta f(\eta) L_{N,T-s} f(\eta) \mu_{T-s}(\eta)  \right\}.
\end{equation}
By Gronwall inequality and \eqref{eq:9} we have
\begin{equation}
  \label{eq:12}
\sum_\eta u(\eta,T)^2 \mu_0(\eta) \le e^{2\int_0^T\Gamma(s) ds + TCN} \sum_\eta u(\eta,0)^2 \mu_T(\eta) = e^{2\int_0^T\Gamma(t) dt + TCN}.
\end{equation}
Then using Schwarz inequality we get
\begin{equation}
  \label{eq:29}
  \bE_{\mu_{0}} \left( e^{\int_0^T \tilde V(\eta(t)) dt}\right) \le \,e^{\int_0^T \Gamma(t) dt + TCN}.
\end{equation}

The Dirichlet forms associated to the generator are 
	\begin{eqnarray}
	\nn
&&\mathfrak{D}_{x,t,\rho}(f)
= \frac 12 \sum_\eta \rho^{1-\eta(x)}(1-\rho)^{\eta(x)} [ f(\eta^{x})- f(\eta)]^2 \mu_t(\eta),\quad x=\pm N
    \\&&\mathfrak{D}_{ex,t}(f)=\frac 12\sum_\eta\sum_{x=-N}^{N-1}
\left(\nabla_{x,x+1} {f(\eta)}\right)^2 \mu_t(\eta)
    	\label{dir1}
		\end{eqnarray}
and we define 
$$
\mathfrak{D}_t(f) =
 \mathfrak{D}_{-N,t,\rho_-(t)}(f) +\mathfrak{D}_{N,t,\rho_+(t)}(f)+ \mathfrak{D}_{ex,t}(f).
$$

Observe that
   \begin{equation}
     \begin{split}
       \sum_\eta f(\eta) (L_{N,t} f)(\eta)
       \mu_t(\eta)=&
       - \frac 12 N^{2+\alpha} \sum_\eta\sum_{x=-N}^{N-1}
       \nabla_{x,x+1} f(\eta)
       \nabla_{x,x+1}(f \mu_t)(\eta) \\ 
       &- N^{2+\alpha} \left(\mathfrak{D}_{N,t,\rho_+(t)}(f) +
         \mathfrak{D}_{-N,t,\rho_-(t)}(f)\right)
     \end{split}
     \label{5.11a}
\end{equation}
and {
\begin{eqnarray}\nn
&&
    \sum_\eta\sum_{x=-N}^{N-1} \nabla_{x,x+1} f(\eta) \nabla_{x,x+1}( f \mu_t)(\eta)  
= \sum_\eta\sum_{x=-N}^{N-1}
\left(\nabla_{x,x+1} f \right)^2 \mu_t (\eta)
 \\&&\hskip2.9cm\nn
+  \sum_\eta\sum_{x=-N}^{N-1} 
f(\eta^{x,x+1}) \nabla_{x,x+1} f (\eta)
\nabla_{x,x+1}\mu_t(\eta)
\\ &&\hskip2.7cm =2\mathfrak{D}_{ex,t}(f) + \frac 1N \tilde B_N(t) {\
  +\ O(N^{-1})} 
  \label{5.11b}
\end{eqnarray}
where}
\begin{equation*}
\tilde B_N(t) = \sum_\eta\sum_{x=-N}^{N-1}
f(\eta^{x,x+1})\nabla_{x,x+1}f(\eta) 
  \left(\eta(x)- \eta(x+1)\right) B(\frac xN,t)
 \mu_t (\eta)
\end{equation*}
{and 
\begin{equation}
  \label{eq:11}
 B(\frac xN,t) = \frac{\bar \rho'(\frac xN,t)}{\bar \rho(\frac  xN,t) (1-\bar \rho(\frac xN,t))},\qquad \| B \|_\infty\le c\|\rho'\|_\infty
\end{equation}
where $c>0$ since  $\rho_\pm(t)$ are uniformly away from 0 and 1.}
By an elementary inequality we have: 
\begin{eqnarray*}
    \left|\tilde B_N(t) \right| &&\le \frac N2
    \mathfrak{D}_{ex,t}(f) + \frac 2{2N} \sum_\eta\sum_{x=-N}^{N-1}
 {f}(\eta^{x,x+1})^2  \left(\eta(x)- \eta(x+1)\right)^2 B(\frac xN,t)^2
 \mu_t(\eta) \\&&
 \le \frac N2
    \mathfrak{D}_{ex,t}(f) + \frac 1{N}
   \sum_{x=-N}^{N-1} B(\frac xN,t)^2  \sum_\eta {f^2} (\eta) \mu_t (\eta^{x,x+1}) \end{eqnarray*}
{ Since $\mu_t (\eta^{x,x+1})=F(x,t)\mu_t(\eta)$ with $F(x,t)$ 
a uniformly bounded function of $\bar \rho(\frac xN,t)$ and $\bar \rho(\frac {x+1}N,t)$ and $\eta$,
using \eqref{eq:11} we get}
\begin{equation}
  \left|\tilde B_N(t)\right| \le \frac N2 \mathfrak{D}_{ex}(f) + C \le \frac N2 \mathfrak{D}_t(f) + C
  \label{5.11c}
\end{equation}
From \eqref{5.11a}, \eqref{5.11b} and \eqref{5.11c} we get
\begin{equation}
  \label{eq:32}
  \begin{split}
    \Gamma(s) \le \sup_f \left\{ N^\alpha a \sum_\eta V_{N,\ve}(T-s, \eta)
      f^2(\eta) \mu_{T-s}(\eta) - N^{2+\alpha}\frac 12 \mathfrak{D}_{T-s}(f)
    \right\} + N^{1+\alpha} C\\
    \le N^{1+\alpha} \sup_f \left\{ \frac aN \sum_\eta V_{N,\ve}(T-s,\eta)
      f^2(\eta) \mu_{T-s}(\eta) - N \frac 12 \mathfrak{D}_{T-s}(f)
    \right\} + N^{1+\alpha}  C
  \end{split}
\end{equation}
Then we are left to prove that for any $a$ and any $t>0$:
\begin{equation}
  \label{eq:33}
   \limsup_{\ve\to 0} \limsup_{N\to\infty} \sup_f 
   \left\{  \frac 1N  \sum_\eta aV_{N,\ve}(t, \eta) f^2(\eta) \mu_t(\eta) - \frac N2 \mathfrak{D}_t(f) \right\} \le C'
\end{equation}
This follows by proving that, for a suitable constant $C$
\begin{equation}
  \label{eq:19}
   \limsup_{\ve\to 0} \limsup_{N\to\infty} \sup_{ \mathfrak{D}_t(f) \le CN^{-1}} 
   \frac 1N  \sum_\eta V_{N,\ve}(t, \eta) f^2(\eta) \mu_t(\eta) \le 0
\end{equation}
The rest of the proof is identical as in theorem 3.1 in \cite{dmo1} (see pages 1051-2).
\end{proof}

With a similar argument follows also the super-exponential control of the densities at the boundaries:

\begin{prop} For any $\delta >0$ we have
  \begin{equation}
    \begin{split}
      \limsup_{\ve\to 0} \limsup_{N\to\infty} \frac {1}{N^{1+\alpha}} \log  \bP_\eta
      \left( \int_0^T \left| \frac 1{\ve N} \sum_{x= \pm N}^{\pm N(1\mp\ve)} \eta_t(x) -
          \rho_{\pm}(t)\right| \ge \delta\right) = -\infty
    \end{split}
\label{eq:43}
\end{equation}

\end{prop}

\section{The exponential martingales}
\label{sec:expon-mart-}

We use the following notations: for $x=-N, \dots, N-1$
\begin{align}
	\label{eq:mob}
	\phi_-(\eta,x)  =\eta(x+1)(1-\eta(x)), \quad
	\phi_+(\eta,x)   =\eta(x)(1-\eta(x+1)),
\end{align}
and at the boundaries
\begin{equation}
  \label{eq:20}
  \begin{split}
    \phi_-(\eta,-N-1,t) &:= (1-\rho_{-}(t))\eta(-N)\\
    \phi_+(\eta,-N-1,t) &:= \rho_{-}(t))(1-\eta(-N))\\
    \phi_-(\eta,N,t) &:= \rho_{+}(t)(1-\eta(N))\\
    \phi_+(\eta,N,t) &:= (1-\rho_{+}(t))\eta(N) .
  \end{split}
\end{equation}

Given two functions $z_\pm (t,y)$, 
we associate the following exponential martingales, for $x=-N-1, \dots, N$
\begin{equation}
  \begin{split}
    \label{eq:17}
    &\mathcal E_\pm(z_\pm, x, T) = \\ 
    &\exp\Big\{  \int_0^T z_\pm(t,x/N)\; d\h_\pm(t,x) 
      - \int_0^T N^{2+\alpha}\left(e^{z_\pm(t,x/N)} -1\right) 	\phi_\pm(\eta_t,x)\; dt \Big\}\\
    &=  
    \exp\Big\{  z_\pm(T,x/N)\; \h_\pm(T,x)\\
    & \qquad   - \int_0^T \Big[ \partial_t z_\pm(t,x/N) \h_\pm(t,x) + 
      N^{2+\alpha} \left(e^{z_\pm(t,x/N)} -1\right) 	\phi_\pm(\eta_t,x) \Big] \; dt \Big\}
  \end{split}
\end{equation}
We now choose a smooth function $H(t,y), y\in [-1,1]$ such that $H(t,-1) = 0$ and we set 
\begin{equation}
z_+(t,y) ={H(t,y+1/N)-H(t,y)} = - z_-(t,y), \qquad -1\le y \le 1-1/N\label{eq:85}
\end{equation}
and at the boundaries 
\begin{equation}\label{eq:85b}
  z_+(t,-1-1/N) : = \frac 1N \partial_yH(t,-1), \quad   z_+(t,1) : = \frac 1N \partial_yH(t,1). 
\end{equation}

The martingales defined by \eqref{eq:17} are orthogonal, 
consequently taking the product we still have an exponential martingale equal to
\begin{equation}
  \label{eq:21}
  \begin{split}
    &\prod_{x=-N-1}^N \prod_{\sigma=\pm} \mathcal E_\sigma(z_\sigma, x, T) = \\
   & \exp\Big\{ \sum_{x=-N-1}^N \Big[ z_+(T,x/N)\; \h(T,x) - \int_0^T  \partial_t z_+(t,x/N) \h(t,x) dt \\
    &- N^{2+\alpha} \int_0^T \left(\left(e^{z_+(t,x/N)}-1\right)  \phi_+(\eta_t,x) + 
      \left(e^{-z_+(t,x/N)}-1\right)\phi_-(\eta_t,x) \right) \; dt\Big] \;\Big\}
  \end{split}
\end{equation}
For large $N$ we use Taylor approximation, for $x= - N, \dots, N-1$,
\begin{equation*}
  \begin{split}
    &\left(e^{z_+(t,x/N)}-1\right) \phi_+(\eta_t,x) +
    \left(e^{-z_+(t,x/N)}-1\right)\phi_-(\eta_t,x) \\
    &= -z_+(t,x/N)
    \left(\eta_t(x+1) - \eta_t(x)\right) + \frac 12 z_+(t,x/N)^2
    (\phi_+(\eta_t,x) + \phi_-(\eta_t,x) ) + O({z_+(t,x/N)^3}).
  \end{split}
\end{equation*}
and on the boundaries
\begin{equation*}
  \begin{split}
    &\left(e^{z_+(t,-1-1/N)}-1\right) \phi_+(\eta_t,-N-1,t) +
    \left(e^{-z_+(t,-1-1/N)}-1\right)\phi_-( \eta_t,-N-1,t) \\
    &= -z_+(t,-1-1/N)
    \left(\eta_t(-N) - \rho_-(t) \right) \\
    &+ \frac 12 z_+(t,-1-1/N)^2
    (\phi_+( \eta_t,-N-1,t) + \phi_-( \eta_t,-N-1,t) ) + O({z_+(t,-1-1/N)^3}).\\
  \end{split}
\end{equation*}

\begin{equation*}
  \begin{split}
    &\left(e^{z_+(t,1)}-1\right) \phi_+( \eta_t,N) +
    \left(e^{-z_+(t,1)}-1\right)\phi_-( \eta_t,N) \\
    &= - z_+(t,1)\left(\rho_+(t) - \eta(t, N)  \right) + \frac 12 z_+(t,1)^2
    (\phi_+( \eta_t,N) + \phi_-( \eta_t,N) ) + O({z_+(t,1)^3}).
  \end{split}
\end{equation*}

We can rewrite the exponential martingale \eqref{eq:21} as
\begin{equation}
  \label{eq:22}
  \begin{split}
      &\prod_{x=-N-1}^N \prod_{\sigma=\pm} \mathcal E_\sigma(z_\sigma, x, T) = \\
   & \exp\Big\{ \sum_{x=-N-1}^N \Big[ z_+(T,x/N)\; \h(T,x) - 
   \int_0^T  \partial_t z_+(t,x/N) \h(t,x) dt \big]\\
    & + N^{2+\alpha} \int_0^T \sum_{x=-N-1}^{N}\Big[ z_+(t,x/N) \left(\eta_t(x+1) - \eta_t(x)\right) \\
    &\qquad\qquad - \frac 12 z_+(t,x/N)^2 (\phi_+( \eta_t,x) + \phi_-( \eta_t,x) ) \Big] dt + O(N^{\alpha})\;\Big\}
  \end{split}
\end{equation}
where in the above expression we set the convention
 $\eta_t(N+1) := \rho_+(t)$ and $\eta_t(-N-1) := \rho_-(t)$.

After a summation by parts we have
\begin{equation}
  \label{eq:23}
  \begin{split}
     & \exp\Big\{ \sum_{x=-N-1}^N \Big[ z_+(T,x/N)\; \h(T,x) - 
   \int_0^T  \partial_t z_+(t,x/N) \h(t,x) dt \big]\\
    &+ N^{2+\alpha} \int_0^T \sum_{x=-N}^{N} \Big[ \left(z_+(t,\frac{x-1}N) - z_+(t,\frac{x}N)\right) \eta_t(x)\\
    &\qquad\qquad - \frac 12 z_+(t,x/N)^2 (\phi_+( \eta_t,x) + \phi_-( \eta_t,x) ) \Big] dt \\
    &\qquad\qquad  +N^{2+\alpha} \int_0^T \left( z_+(t,1) \rho_+(t) - z_+(t,-1-1/N) \rho_-(t) \right) dt
    + O(N^{\alpha})\;\Big\}
  \end{split}
\end{equation}
Using \eqref{eq:claw}
%
%
 and  the special expression for $z_+$, we have for the logarithm of \eqref{eq:23}
\begin{eqnarray*}
&& \sum_{x=-N}^N H(T,x/N) \left[\eta_T(x)   -  \eta_0(x) \right] 
  \\&&\qquad  + \left[H(T,1) + \frac 1N \partial_y H(t,1)\right] \h(T,N)
       +\frac 1N\partial_y H(T,-1) \h(T,-N-1) 
   \\&& - \int_0^T \sum_{x=-N}^N \partial_t H(t,x/N) [\eta_t(x) - \eta_0(x)] dt 
  \\&&- \int_0^T \left(\left[\partial_t H(t,1) + \frac 1N \partial_t\partial _y H(t,1)\right] \h(t,N)
       + \frac 1N \partial_t \partial_y H(t,-1) \h(t,-N-1)\right)  \; dt\\
  &&- N^{\alpha} \int_0^T\Big( \sum_{x=-N-1}^{N} \Big[ \partial_{yy} H(t,\frac{x}N) \eta_t(x)
     + \frac 12 \left(\partial_y H(t,\frac{x}{N})\right)^2 (\phi_+( \eta_t,x) + \phi_-( \eta_t,x) ) \Big]\\
  &&\qquad
     + N \left[\partial_yH(t,1) \rho_+(t) - \partial_yH(t,-1) \rho_-(t)\right]\Big)dt + O(N^{\alpha})
\end{eqnarray*}
Since we are interested only in the terms that have order $N^{1+\alpha}$, 
we can forget all terms of order $N$ in the above expression,
and the exponential martingale has the form:
\begin{equation}
  \label{eq:47}
  \begin{split}
    \exp &\Bigg[ N^{1+\alpha} \Big\{ H(T,1) \h_N(T,1)  - \int_0^T \partial_t H(t,1) \h_N(t,1) \; dt\\
    &- \int_0^T\Big(\frac 1N  \sum_{x=-N-1}^{N}  \partial_{yy} H(t,\frac{x}N) \eta_t(x)
  + \partial_yH(t,1) \rho_+(t) - \partial_yH(t,-1) \rho_-(t) \\
   & + \frac 1{2N} \sum_{x=-N-1}^{N} 
   \left(\partial_y H(t,\frac{x}{N})\right)^2 (\phi_+(\eta_t,x) + \phi_-(\eta_t,x) )  \Big)dt + O(N^{-(\alpha\wedge 1)})
   \Big\}\Bigg].
  \end{split}
\end{equation}

After the superexponential estimate proved in the previous section,
that also fix the densities at the boundaries,
it follows the variational formula for the rate function
given by \eqref{eq:24}.

\section{The upper bound}
\label{sec:upper-bound}

\subsection{Exponential tightness}
\label{sec:expon-tightn}

The  following proposition uses standard arguments and we give a proof for completeness in Appendix A:

\begin{prop}\label{exp-tight}
  There exist a sequence of compact sets $K_L$ in $\mathcal D([0,T],\bR)$ such that
  \begin{equation}
    \label{eq:69}
    \lim_{L\to\infty} \lim_{N\to\infty} \frac 1{N^{1+\alpha}} \log \mathbb P_{\eta} \left( \h_N(1) \in K_L^c \right) = -\infty. 
  \end{equation}
\end{prop}

\subsection{Proof of the upper bound}
\label{sec:proof-upper-bound}

By \eqref{eq:47} we can rewrite the exponential martingale defined by \eqref{eq:22} as
\begin{equation}
  \label{eq:71}
  \begin{split}
    M^H_N = &\exp\Big[ N^{1+\alpha} \Big\{ H(T,1) \h_N(T,1)  - \int_0^T \partial_t H(t,1) \h_N(t,1) \; dt\\
      &- \int_0^T\Big(\int_{-1}^1 \pi_N(t,y) \partial_{yy} H(t,y) dy
      + \partial_yH(t,1) \rho_+(t) - \partial_yH(t,-1) \rho_-(t) \\
      & + \frac 1{2N} \sum_{x=-N-1}^{N} \left(\partial_y
        H(t,\frac{x}{N})\right)^2 (\phi_+(\eta_t, x) + \phi_-(\eta_t, x))
      \Big)dt + O(N^{-(\alpha\wedge 1)})\Big\}\Big].
    \end{split}
\end{equation}
where
\begin{equation*}
  \lim_{N\to\infty} \mathbb E^{\bP_\eta}\left(O(N^{-(\alpha\wedge 1)})\right) = 0. 
\end{equation*}
After applying the superexponential estimates of section \ref{sec:super-estim}, we have
\begin{equation*}
   \begin{split}
    M^H_N = &\exp\Big[ N^{1+\alpha} \Big\{ H(T,1) \h_N(T,1)  - \int_0^T \partial_t H(t,1) \h_N(t,1) \; dt\\
      &- \int_0^T\Big(\int_{-1}^1 \pi_N(t,y) \partial_{yy} H(t,y) dy
      + \partial_yH(t,1) \rho_+(t) - \partial_yH(t,-1) \rho_-(t) \\
      & + \int_{-1}^1 \left(\partial_y H(t,y)\right)^2 \phi(\bar\eta_{N,\ve}(t,y)) dy \Big)dt 
      + O_{\ve, N} + O(N^{-(\alpha\wedge 1)})\Big\}\Big].
    \end{split}
\end{equation*}

Then for every set $A \subset \Xi$ we have
\begin{equation*}
  1 = \mathbb E^{\bP_\eta}
  \left(  M^H_N \right) \ge e^{N^{1+\alpha}\sup_H \inf_{(J,\rho) \in A}  I(H, J, \rho)} \bP_\eta \left( (\h_N(1), \pi_N) \in A \right)
\end{equation*}
where
\begin{equation*}
  \begin{split}
    I(H,J,\rho) = &H(T,1) J(T)  - \int_0^T \partial_t H(t,1) J(t) \; dt\\
    &- \int_0^T\Big(\int_{-1}^1 \rho(t,y) \partial_{yy} H(t,y) dy
    + \partial_yH(t,1) \rho_+(t) - \partial_yH(t,-1) \rho_-(t) \\
    & \qquad + \int_{-1}^1 \left(\partial_y H(t,y)\right)^2 \phi(\rho(t,y)) dy \Big)dt.
  \end{split}
\end{equation*}
Using the lower semicontinuity of $I(J,\rho)$ and a standard argument 
(see \cite{Var}, lemma 11.3 or \cite{KLbook} lemma { A2.3.3}) 
we have for a compact set $C\subset \Xi$:
\begin{equation*}
  \limsup_{N\to\infty} \frac{1}{N^{1+\alpha}} \log \bP_\eta\left( (\h_N(1), \pi_N) \in C \right) \le - \inf_{(J,\rho) \in C} I(J,\rho).
\end{equation*}
The extension to closed set follows from the exponential compactness {proved in \eqref{eq:69}, see  \cite{KLbook} pag.271}.

\section{The lower bound}
\label{sec:lower-bound}

The proof of the lower bound follows a standard argument, consequently we will only sketch it here, 
since all the ingredients are  already proven.
It is enough to prove that given $(J, \rho)\in \Xi$ such that $I(J,\rho)<\infty$, 
then for any open neighbor $\mathcal O$ of it we have
\begin{equation}
  \label{eq:54}
  \liminf_{N\to\infty} \frac 1{N^{1+\alpha}} \log \bP_\eta \left((\h_N(1), \mu_N) \in \mathcal O\right) \ \ge \ - I(J,\rho).
\end{equation}
By proposition \ref{prop2}, we can assume $J$ and $\rho$ such that $J'(t)$ exists and is bounded, $\rho$ bounded away from $0$ and $1$,  and $\partial_y\rho(t,y)$ 
exists and is bounded.
Then we consider the weakly asymmetric  exclusion dynamics with drift given by 
\begin{equation}
  \label{eq:55}
  \partial_y \bar H(t,y) = \frac 12 \frac{J'(t) + \partial_y\rho(t,y)}{\phi(\rho(t,y))},
\end{equation}
more precisely, recalling the definition of $z(t,y)$ given by \eqref{eq:85} and \eqref{eq:85b},
 the jump rate from $x$ to $x+1$ at time $t$ is taken to be $N^{2+\alpha}e^{ z (t,x/N) }$,
%
and from $x+1$ to $x$ is given by $N^{2+\alpha}e^{-z (t,x/N) }$,
while at the boundaries the birth rates are given by $N^{2+\alpha}e^{\frac 1N \partial_y H (t,\pm 1) }$ and
the death rates by $N^{2+\alpha}e^{-\frac 1N \partial_y H (t,\pm 1) }$.

We call $Q_N$ the law of this weakly asymmetric process that start with the same initial condition $\eta$.
The Radon-Nykodyn derivative $\frac{dQ_N}{d\mathbb P_\eta}$ is given by \eqref{eq:21}. 

The quasi static limit for this process is the following:

\begin{prop}\label{wahl}
Let $\tilde Q_N$ the law on $\Xi$ of $(\h_N(1), \pi_N)$ under $Q_N$, then 
\begin{equation}
  \label{eq:61}
  \tilde Q_N \longrightarrow \delta_{(J,\rho)}
\end{equation}
 \end{prop}
 
 \begin{proof}[Proof of proposition \ref{wahl}]
By \eqref{eq:63} we can extend the superexponential estimates contained in section \ref{sec:super-estim} to 
$Q_N$. {In fact 
we have that}
\begin{equation*}
  \mathbb E^{\bP_\eta} \left( \left(\frac{dQ_N}{d\bP_\eta}\right)^{2}\right)^{1/2} \le e^{cN^{1+\alpha}}
\end{equation*}
and by Schwarz inequality
\begin{equation*}
  Q_N \left( A_{N,\ve}\right) \le \mathbb P_\eta \left( A_{N,\ve}\right) e^{cN^{1+\alpha}}
\end{equation*}
where $A_{N,\ve} = \{\int_0^T V_{N,\ve}(t,\eta_t) dt \ge N\delta\}$, 
and \eqref{eq:28} extends immediately to $Q_N$.
At this point the proof of the quasi-static hydrodynamic limit follows similar to the one in \cite{dmo1}. \end{proof}

\bigskip{ We then write
\begin{equation*}
  \bP_\eta\left((\h_N(1), \pi_N) \in \mathcal O\right)=
  \mathbb E^{Q_N}\Big(\frac{d\mathbb P_\eta}{dQ_N}\mathbf 1_{(\h_N, \pi_N) \in \mathcal O}\Big)
\end{equation*}
Since $\mathcal O$ contains $(J,\rho)$, by Proposition \ref{wahl}, under $Q_N$ the probability of the event 
$(\h_N(1), \pi_N) \in \mathcal O$ is close to one. By Jensen inequality
\begin{equation*}
  \frac 1{N^{1+\alpha}} \log \mathbb P_\eta \left((\h_N(1), \pi_N) \in \mathcal O\right) \ \ge
  - \mathbb E^{Q_N}\Big( \frac 1{N^{1+\alpha}}\log\frac{dQ_N}{d\mathbb P_\eta}\Big)=-\frac   
1{N^{1+\alpha}}H(Q_N|\bP_\eta) 
\end{equation*}
where $H(Q_N|\bP_\eta)$ is the relative entropy of $Q_N$ with respect to $\bP_\eta$.

\bigskip
The lower bound is  then a consequence of the following Proposition.}
 \begin{prop}\label{relent}
Let $H(Q_N|\bP_\eta) = \mathbb E^{Q_N}\left(\log \frac{dQ_N}{d\bP_\eta}\right)$ 
the relative entropy of $Q_N$ with respect to $\bP_\eta$. Then 
   \begin{equation}
  \label{eq:59}
  \lim_{N\to\infty} \frac 1{N^{1+\alpha}} H(Q_N|\bP_\eta) = I(J,\rho),
\end{equation}
 \end{prop}





The proof of proposition \ref{relent} is a direct consequence of \eqref{eq:21} and of Proposition \ref{wahl}.

\section{Appendix A: The exponential tightness}
\label{sec:append-a:-expon}

We prove here proposition \ref{exp-tight}. The arguments used here are just slight variations 
of the standard ones (e.g. section 10.4 in  \cite{KLbook}).
{SInce $\mathcal M$ is compact, we have only to control that the distribution of $\h_N(1,t)$ is  exponentially tight.}
This is consequence of the 
following 2 propositions.

\begin{prop}\label{sexp-h}
  \begin{equation}
    \lim_{L\to\infty} \lim_{N\to\infty} \frac 1{N^{1+\alpha}} \log \mathbb \bP_\eta \left(\sup_{0\le t\le T} |\h_N(t,1)| \ge L\right)
    = -\infty. 
\label{eq:63}
  \end{equation}
\end{prop}

\begin{prop} For any $\ve>0$:
  \label{sexp-modh}
\begin{equation}
  \lim_{\delta\to 0} \lim_{N\to\infty} \frac 1{N^{1+\alpha}} \log
  \mathbb \bP_\eta \left(\sup_{|t-s|\le \delta} |\h_N(t,1) - \h_N(s,1)| \ge \ve \right) = -\infty. 
\label{eq:63-c}
  \end{equation}
\end{prop}

\begin{proof}[Proof of Proposition \ref{sexp-h} ]
  Since the difference between $\h_N(t,1)$ and $\bar\h_N(t) := \frac 1{2N} \sum _x \h_N(t,x/N)$ 
is uniformly small, we have just to prove it for $\bar\h_N(t)$.
  For $\beta\in \bR$, consider the  exponential martingale \eqref{eq:21} 
with $z_+(t,x) = \frac \beta N,\ x=-N, \dots, N-1$ and $z_+(t,-N-1) =0$. 
  This is given by
  \begin{equation*}
    \begin{split}
      M_t = &\exp\Big\{ N^{1+\alpha} \left(\beta\; \bar\h_N(t) - A_N(\beta,t)\right) \Big\}\\
     A_N(\beta,t) = &  N \int_0^t \sum_{x=-N}^{N-1}\left( (e^{\beta/N} - 1) \phi_+(\eta_s,x) + (e^{-\beta/N} - 1) \phi_-(\eta_s,x)\right)  ds  
    \end{split}
  \end{equation*}
  {Notice that, by expanding the exponentials and using the explicit form of $\phi_\pm$, we have
  that} $0\le A_N(\beta,t) \le C T (|\beta| + \beta^2)$ for some constant $C$.
Then for any $\beta>0$ we have, by Doob's inequality,
\begin{equation}\label{eq:A1}
  \begin{split}
    \bP_\eta\left( \sup_{0\le t\le T} |\bar\h_N(t)| \ge L\right) &\le 
    \bP_\eta\left( \sup_{0\le t\le T} |\log M_t| \ge N^{1+\alpha} \left(\beta L- C T (\beta + \beta^2)\right)\right) \\
    &\le \bP_\eta\left( \sup_{0\le t\le T} \log M_t \ge N^{1+\alpha} \left(\beta L - C T (\beta + \beta^2)\right)\right)\\
    &= \bP_\eta\left( \sup_{0\le t\le T} M_t \ge e^{N^{1+\alpha} \left(\beta L- C T (\beta + \beta^2)\right)}\right)\\
   & \le e^{- N^{1+\alpha} \left(\beta L- C T (\beta+ \beta^2)\right)},
  \end{split}
\end{equation}
that concludes the proof.
\end{proof}

\begin{proof}[Proof of Proposition \ref{sexp-modh} ]
Since
\begin{equation*}
  \begin{split}
    \left\{\sup_{|t-s|\le \delta} |\bar\h_N(t) - \bar\h_N(s)| \ge \ve \right\} \subset 
    \bigcup_{k=0}^{[T\delta^{-1}]} \left\{\sup_{k\delta \le t\le (k+1) \delta} |\bar\h_N(t) - \bar\h_N(k\delta)| \ge \ve/4 \right\}
  \end{split}
\end{equation*}
Since
\begin{equation*}
  \begin{split}
    \log \mathbb \bP_\eta \left(\sup_{|t-s|\le \delta} |\bar\h_N(t) - \bar\h_N(s)| \ge \ve \right) \le \max_{k} 
    \log \mathbb \bP_\eta \left(\sup_{k\delta \le t\le (k+1) \delta} |\bar\h_N(t) - \bar\h_N(k\delta)| \ge \ve/4 \right) \\
    + \log([T\delta^{-1}])
  \end{split}
\end{equation*}
By the same estimate made in \eqref{eq:A1} we have
\begin{equation*}
  \log \mathbb \bP_\eta \left(\sup_{k\delta \le t\le (k+1) \delta} |\bar\h_N(t) -\bar\h_N(k\delta)| \ge \ve/4  \right) \le 
  - N^{1+\alpha} \left(\beta \ve/4 - C\delta (\beta + \beta^2)\right)
\end{equation*}
and with a proper choice of $\beta$ we get the following bound with a constant $C'$ independent of $k$: 
\begin{equation*}
  \frac 1{N^{1+\alpha}}  \log \mathbb \bP_\eta \left(\sup_{k\delta \le t\le (k+1) \delta} |\bar\h_N(t) -\bar \h_N(k\delta)|
    \ge \ve/4  \right) \le -\frac{C'\ve^2}{\delta}.
\end{equation*}
\end{proof}

\bibliographystyle{amsalpha}

\begin{thebibliography}{99}

\bibitem{bertiniprl}
 Bertini L., A. De Sole, Gabrielli D., Jona-Lasinio G., and Landim C.,
\emph{Macroscopic current fluctuations in stochastic lattice gases}, 
 Physical Review letters,  \textbf{94}:030601, (2005).

\bibitem{bertinijsp2016} 
 Bertini L., A. De Sole, Gabrielli D., Jona-Lasinio G., and Landim C.,
\emph{Non Equilibrium Current Fluctuations in Stochastic Lattice Gases},
Journal of Statistical Physics, \textbf{123}, No. 2, 237-276, (2006), DOI: 10.1007/s10955-006-9056-4


\bibitem{bertinirmp}
 Bertini L., A. De Sole, Gabrielli D., Jona-Lasinio G., and Landim C.,
\emph{Macroscopic Fluctuation Theory}, Review of Modern Physics, \textbf{87}, 593--636, (2014).

\bibitem{bgjll13}
 Bertini L., Gabrielli D., Jona-Lasinio G., and Landim C., \emph{Clausius
  inequality and optimality of quasistatic transformations for nonequilibrium stationary
states}. Physical Review Letters,  \textbf{110}(2):020601, (2013). 

\bibitem{BLM}
Bertini L., Landim C., Mourragui M., \emph{Dynamical large deviations for the 
boundary driven weakly asymmetric exclusion process}, 
Annals of Probability, \textbf{37}, 6, 2357-2403, (2009).

\bibitem{bd04} 
Bodineau T., Derrida B., \emph{Current Fluctuations in Nonequilibrium Diffusive Systems: An Additivity Principle},
 Physical Review Letters, \textbf{92}, 180601, (2004).

\bibitem{bd06} 
Bodineau T., Derrida B., \emph{Current Large Deviations for Asymmetric Exclusion Processes with 
Open Boundaries}, Journal of Statistical Physics, \textbf{123}, n.2, 277-- 300 (2006). 

\bibitem{ch-landim}
  Chavez E., Landim C., \emph{A Correction to the Hydrodynamic Limit 
    of Boundary Driven Weakly Asymmetric Exclusion Processes in a Quasi-Static Time Scale},
  Journal of Statistical Physics, \textbf{163}, Issue 5, pp 1079–1107, (2016). 

 \bibitem{DPTV2}
 De Masi A., Presutti E., Tsagkarogiannis D., and Vares M.E., \emph{Truncated correlations in the stirring process with
   births and deaths }, Electronic Journal of Probability \textbf{17} 1-35 (2012)


\bibitem{dmo1} De Masi A., Olla S., \emph{Quasi Static Hydrodynamic Limits},  Journal of
  Statistical Physics, \textbf{161}, n.5, 1037--1058 (2015). 

\bibitem{GKMP}   Galves A.,  Kipnis C., Marchioro C. and
  Presutti E., \emph {Non equilibrium measures which exhibt a
    temperature gradient: study of a model. }  Commun. Math. Phys.,
 \textbf{819}, 127-1474 (1981). 

\bibitem{KLbook} Kipnis C.,  Landim C.,
 \emph {Scaling Limits of Interacting Particle Systems}, Grundlehren der mathematischen
Wissenschaften 320, Springer-Verlag, Berlin, New York, (1999).


\bibitem{KOV1989}
 Kipnis C.,  Olla S., and Varadhan  S.R.S.,   \emph{Hydrodynamics
and large deviation for simple exclusion processes}, Commun. Pure Appl. Math. \textbf{42}  115 --137(1989). 

\bibitem{LO2015}  Letizia V., and Olla S.,  \emph{ Non-equilibrium isothermal transformations in a 
  temperature gradient from a macroscopic dynamics}, {Annals of Probability}, \textbf{45}, N. 6A (2017), 3987--4018.
 doi:10.1214/16-AOP1156. 

\bibitem{olla2014}
 Olla S., \emph{Microscopic derivation of an isothermal thermodynamic
  transformation}, In From Particle Systems to Partial Differential
Equations, pages 225--238. Springer, 2014. 

\bibitem{Var} Varadhan, S.R.S.,
\emph{Large Deviations and Applications}, 
CBMS-NSF Regional Conference Series in Applied Mathematics 46, Philadelphia, 
Society for Industrial and Applied Mathematics, 1984. 

\end{thebibliography}

\end{document}